\documentclass[12pt]{amsart}
\usepackage{geometry}   
\usepackage[colorlinks,citecolor = red, linkcolor=blue,hyperindex]{hyperref}
\usepackage{euscript,eufrak,verbatim, mathrsfs}
\usepackage[psamsfonts]{amssymb}
\usepackage{bbm}
\usepackage{graphicx}
 \usepackage{float}
\usepackage{float, tikz}

 \usepackage{extarrows}
 \usepackage[all, cmtip]{xy}

\usepackage{upref, xcolor, dsfont}
\usepackage{amsfonts,amsmath,amstext,amsbsy, amsopn,amsthm}
\usepackage{enumerate}

\usepackage{url}

\usepackage{mathtools}
\usepackage{bookmark}

 \usepackage{euscript}
\usepackage{helvet}         
\usepackage{courier}        
\usepackage{type1cm}        
\usepackage{multicol}        
\usepackage[bottom]{footmisc}

\newtheorem{theorem}{Theorem}[section]
\newtheorem*{theorem*}{Theorem A}
\newtheorem{lemma}[theorem]{Lemma}

\newtheorem*{definition*}{Definition}
\newtheorem*{remark*}{Remark}

\newtheorem*{observation*}{Observation}

\newtheorem*{assumption*}{Assumption}

\theoremstyle{definition}

\theoremstyle{remark}
\newtheorem{remark}{Remark}[section]
\newtheorem{claim}{Claim}[section]

\newcommand{\R}{\mathbb{R}}
\newcommand{\N}{\mathbb{N}}

\newcommand{\D}{\mathbb{D}}
\newcommand{\C}{\mathbb{C}}

\newcommand{\B}{\mathbb{B}}

\begin{document}

\title{The weak-type (1,1) estimate of the $\mathcal{H}$-Harmonic Bergman projection}

\author
{Kenan Zhang}
\address
{Kenan Zhang\\
	School of Mathematical Sciences, Fudan University\\
	Shanghai, 200433, China}
\email{knzhang21@m.fudan.edu.cn}

\thanks{}

\begin{abstract}
	In this note, the author recalls the Calderon-Zygmund theory on the unit ball and derives the weak (1,1) boundedness of the projection for $\mathcal{H}$-harmonic Bergman space.
\end{abstract}


\maketitle

\setcounter{equation}{0}

\section{Introduction}
For $n\geqslant 2$, let $\B_n=\{x\in\mathbb{R}^n:|x|<1\}$ be the unit ball and $\partial\B_n=\{x\in\R^n:|x|=1\}$ be the unit sphere, where $|x|=(\sum_{i=1}^n|x_i|^2)^{1/2}$ is the standard Euclidean norm on $\mathbb{R}^n$. 
The hyperbolic ball is the unit ball $\B_n$ equipped with the hyperbolic metric
\[ds^2=\frac{4}{(1-|x|^2)^2}\sum\limits_{i=1}^n dx_i^2.\]

For $a\in\B_n$, the M\"{o}bius transformation $\varphi_a$ exchanging $a$ and $0$ in $\B_n$ is the following self-map on $\B_n$:  
\[
\varphi_a(x)=\frac{a|x-a|^2+(1-|a|^2)(a-x)}{|x-a|^2+(1-|a|^2)(1-|x|^2)}.
\] 
The Laplacian operator $\Delta_h$ with respect to the hyperbolic metric is the linear operator 
such that
\begin{equation}\label{1.1}
	(\Delta_hf)(a)=\Delta(f\circ\varphi_a)(0)
\end{equation}
for all $f\in C^2(\B_n)$ and for each $a\in\B_n$.
This operator is also known as the Laplace-Beltrami operator or the invariant Laplace operator. By calculation, it is easy to show that
\[(\Delta_hf)(a)=(1-|a|^2)^2\Delta f(a)+2(n-2)(1-|a|^2)\langle a,\nabla f(a)\rangle,\]
where 
\[\Delta=\frac{\partial^2}{\partial x_1^2}+\cdots+\frac{\partial^2}{\partial x_n^2},\quad\nabla=\left(\frac{\partial}{\partial x_1},\cdots,\frac{\partial}{\partial x_n}\right)\]
denote the standard Euclidean Laplacian and gradient operators, respectively.

We say that a complex-valued function $f\in C^2(\B_n)$ is $\mathcal{H}$-harmonic if $(\Delta_hf)(x)=0$ for all $x\in\B_n$. The linear space of $\mathcal{H}$-harmonic functions on $\B_n$ is denoted by $\mathcal{H}(\B_n)$. For further background and basic properties of $\mathcal{H}$-harmonic functions see \cite{Ureyen2023} and \cite{Stoll}.

Let $\nu$ be the Lebesgue measure on $\B_n$ normalized so that $\nu(\B_n)=1.$ 
For $1\leqslant p<\infty$, the $\mathcal{H}$-harmonic Bergman space $\mathcal{B}^p $ is the Banach space defined by
\[\mathcal{B}^p =\bigg\{f\in\mathcal{H}(\B_n):\int_{\B_n}|f(x)|^pd\nu (x)<\infty\bigg\}.\]
It is known that the point evaluation functionals $f\mapsto f(x)$ for every $x\in\B_n$ are bounded on $\mathcal{B}^2 $ \cite{Ureyen2023}. Therefore, $\mathcal{B}^2 $ is a reproducing kernel Hilbert space. We denote the reproducing kernel by $\mathcal{R} (x,\cdot)\in\mathcal{B}^2 $. Then, by the reproducing property, we have
\begin{equation}\label{equ1.1}
	f(x)=\int_{\B_n} f(y) \overline{\mathcal{R} (x,y)}d\nu (y),\quad\text{for all } f\in\mathcal{B}^2 .
\end{equation}
Since the reproducing kernel $\mathcal{R} (\cdot,\cdot)$ is real-valued \cite{Ureyen2023}, the conjugation in \eqref{equ1.1} can be removed. The $\mathcal{H}$-Bergman projection $P : L^2(\B_n,d\nu)\to \mathcal{B}^2 $ is given by
\[
(P  f)(x)=\int_{\B_n}f(y) \mathcal{R} (x,y)d\nu (y).
\]
It is straightforward to verify that the $\mathcal{H}$-Bergman projection $P$ is well-defined as an integrable operator for each $x\in\B_n$ and any $f\in L^1(\B_n,d\nu)$.

In fact, the invariant Laplace operator can also be defined as \eqref{1.1} on the Euclidean unit ball of $\C^n$. The space of $\mathcal{M}$-harmonic functions, which are annihilated by the invariant Laplace operator on the complex unit ball, has been studied extensively \cite{ ABC1996, Englis2024,Rudin, Stoll1}. 
However, functional theory of $\mathcal{H}$-harmonic function spaces on the real hyperbolic ball is less well-known. 
The concept of $\mathcal{H}$-harmonic functions was first introduced by Jaming in 1999 \cite{Jaming1999}. 
In the same paper, he investigated the Hardy spaces $H^p(\B_n),0<p<\infty,$ whose elements are the hyperbolic harmonic extensions of distributions belonging to the Hardy spaces $H^p(\partial \B_n)$ of the sphere and obtained an atomic decomposition of these spaces. 
In 2003, Liu and Shi studied the invariant mean-value property of $\mathcal{H}$-harmonic \cite{LiuShi2003}. 
Furthermore, similar spaces defined for holomorphic functions can also be defined for $\mathcal{H}$-harmonic functions, such as Hardy-Sobolev spaces, Lipschitz spaces \cite{Jaming2004}, and weighted Dirichlet spaces \cite{Stoll2012}. 
Additionally, Stoll obtained the reproducing kernel of the Hardy space and radial eigenfunctions of the invariant Laplacian $\Delta_h$ \cite{Stoll2019}.

This note focuses on the Bergman spaces of $\mathcal{H}$-harmonic functions. 
The Bergman projection on reproducing kernel function spaces plays a crucial role in the study of function spaces and their operators. 
Among other things, one of the key issues regarding $P $ is the study of its Calder\'{o}n-Zygmund theory, which includes the $L^p$-boundedness for all $1\leqslant p\leqslant+\infty$. 
In general, related orthogonal projections such as classical Bergman projection on holomorphic and harmonic function space on $\B_n$ are weak-type $(1,1)$ \cite{DHZZ2001}, hence, are bounded on $L^p(\B_n,d\nu)$ for $1<p<+\infty$ by interpolation. 
Indeed, \"Ureyen \cite{Ureyen2023} proved that $P$ is bounded on $L^p(\B_n,d\nu)$ if and only if $1<p<\infty$ and $P:L^\infty\to \mathcal{B}$ is bounded, where $\mathcal{B}$ is the $\mathcal{H}$-harmonic Bloch space on $\B_n$, but leaving the case $p=1$ open. Here the author proves the weak-type $(1,1)$ estimate, closing the gap.
\begin{theorem}\label{th1.1}
	The $\mathcal{H}$-harmonic Bergman projection operator $P$ is of weak-type (1,1), i.e., there exists a constant $C>0$ independent of $f$ such that for any $t>0$
	\begin{equation}\label{equ1.2}
		(P f)_*(t)\leqslant\frac{C}{t}||f||_{L^1(\B_n,d\nu)}, \qquad \text{holds for all $f\in L^1(\B_n,d\nu)$},
	\end{equation}
		where $f_*(t)=\nu (\{x\in\B_n:|f(x)|>t\})$
		is the distribution function of $f$.
\end{theorem}
For a domain $\Omega\subset\R^n$ (or $\C^n$), one can see the harmonic (or holomorphic) Bergman spaces $\mathcal{B}^2(\Omega)$ (or $\mathcal{B}_a^2(\Omega)$) are reproducing kernel Hilbert spaces with kernel $K_\Omega(x,y)$ (or $k_\Omega(z,w)$).
Denote the harmonic Bergman projection by $P_\Omega: L^2(\Omega)\rightarrow\mathcal{B}^2(\Omega)$ and 
\[
	P_\Omega f(x)=\int_\Omega f(y)K_\Omega(x,y) dy
\]
for $f\in L^2(\Omega)$
(or the holomorphic Bergman projection $P_\Omega: L^2(\Omega)\rightarrow\mathcal{B}_a^2(\Omega)$ and 
\[
	P_\Omega f(z)=\int_\Omega f(z)k_\Omega(z,w) dw
\]
for $f\in L^2(\Omega)$.)

Since $P_\Omega$ is an orthogonal projection on $\mathcal{B}^2(\Omega)$ (or $\mathcal{B}_a^2(\Omega)$), it is a bounded operator on $L^2(\Omega)$. Unlike the case of the unit ball, where the Bergman projection is $L^p$-bounded for all $1<p<\infty$, determining the range $1 < p_\Omega \leqslant q_\Omega < \infty$ for which $P_\Omega$ is $L^p$-bounded on a general domain $\Omega$ is a nontrivial question. By interpolation and duality, one finds that $1/p_\Omega + 1/q_\Omega = 1$. Therefore, the $L^p$-boundedness problem of the Bergman projection for a given domain $\Omega$ reduces to determining the exponent $p_\Omega \in (1,2]$ such that $P_\Omega$ extends continuously to $L^p(\Omega)$ for all $p \in (p_\Omega, p_\Omega')$, where $p_\Omega'$ is the Hölder conjugate of $p_\Omega$. 

For a class of domains in complex plane $\C$, Solovyov characterizes the boundedness of the Bergman projection on $L^p(\Omega)$ by geometric properties of the boundary of $\Omega$ \cite{S1,S2}, which are automatically fulfilled for $p\in(\frac{4}{3},4)$.
For a simply connected domain $\Omega$ in $\C$, by conformal mapping, Hedenmalm gives a positive answer for some $p_0\in[\frac{4}{3},2)$ \cite{H}.
Here $p_0$ is independent of $\Omega$.
Edholm and McNeal extend Hedenmalm's result to some $p_\Omega$ for generalized Hartogs triangles, which are pseudoconvex but non-smooth domains in $\C^2$ \cite{EM1,EM2}.
Specifically, Chakrabarti and Zeytuncu show $p_\Omega=\frac{4}{3}$ when $\Omega$ is the classical Hartogs triangle \cite{CZ}.
For higher dimensions, such continuous extension of the Bergman projections also holds on monomial polyhedra in $\C^n$ \cite{BCEM}.

Recall that the classical holomorphic Bergman projection on $\D$ is weak-type $(1,1)$.
One may ask how the Bergman projection $P_\Omega$ behaves on $L^{p_\Omega}(\Omega)$ and $L^{p'_\Omega}(\Omega)$, where $p_\Omega$ is the endpoint of the bounded range of $P_\Omega$.
The works of Huo-Wick \cite{HW} and Christopherson-Koenig \cite{CK1,CK2,CK3} show that $P_\Omega$ is not weak-type $(p,p)$ at the lower endpoint of boundedness, while it is weak-type $(p,p)$ at the upper endpoint for the classical and generalized Hartogs triangles, respectively.
The mapping behavior of the holomorphic Bergman projection on monomial polyhedra at the lower endpoint the is even worse: there is no way to extend the Bergman projection to $L^p(\Omega)$ using its integral representation \cite[Proposition 8.1]{CE}.
In other words, there exists a function $f\in L^p(\Omega)$ such that the integral does not converge.

A similar phenomenon for the harmonic Bergman projection can be observed when $\Omega$ is a punctured smooth domain in $\R^3$.
In such setting, Koenig and Wang show that $P_\Omega$ is bounded on $L^p(\Omega)$ if and only if $\frac{3}{2}<p<3$ \cite{KW}.
For endpoints, $P_\Omega$ is weak-type $(3,3)$ but is not weak-type $(\frac{3}{2},\frac{3}{2})$.
This limitation is missing when $n\neq3$, i.e. the harmonic Bergman projection $P_\Omega$ on such punctured domains in $\R^2$ or $\R^n$, $n\geqslant4$, is always bounded on $L^p(\Omega)$ for $1<p<\infty$.
Besides, $P_\Omega$ is not weak-type $(1,1)$ for $n=2$, while $P_\Omega$ is weak-type $(1,1)$ for $n\geqslant4$.

\section{A dyadic decomposition of the unit ball $\B_n$}
The classical Calder\'on-Zygmund decomposition on $\R^n$ relies on a dyadic system in $\R^n$, which is a useful tool for establishing the weak (1,1) boundedness of singular integrals. 
Similarly, to estimate $P$, we require a dyadic system of the unit ball $\B_n$. 
A collection of sets 
\[
	\mathscr{D}=\{Q_{k,i}\subset\B_n: 1\leqslant k<\infty,1\leqslant i\leqslant M(k)\}
\]
is called a dyadic system on $\B_n$ if 
\begin{enumerate}[{(i)}]
	\item 
	for each integer $k\geqslant 1$, $\{Q_{k,i}\}_{i=1}^{M(k)}$ is a disjoint covering of $\B_n$, i.e.,
	\begin{itemize}
		\item $\B_n=\bigcup\limits_{i=1}^{M(k)} Q_{k,i}$ and 
		\item for all $1\leqslant i\neq j\leqslant M(k), Q_{k,i}\cap Q_{k,j}=\emptyset $.
	\end{itemize}
	\item if $1\leqslant k\leqslant l$, then either $Q_{l,j}\subset Q_{k,i}$ or $Q_{k,i}\cap Q_{l,j}=\emptyset$.
\end{enumerate}

Let $\rho(x,y)=|x-y|$ be the Euclidean metric.
For $x\in\B_n,\ r>0$, denote
\[
	B(x,r)=\{y\in\B_n:|x-y|<r\}.
\]
We say the dyadic system $\mathscr{D}$ is associated with a triple of positive constants $(\eta,\kappa_0,\kappa_1)$ if there exists a point set 
\[
	\mathcal{P}=\{x_{k,i}\in\B_n:1\leqslant k<\infty,1\leqslant i\leqslant M(k)\}
\] 
such that 
\begin{equation*}
	B(x_{k,i},\kappa_0\eta^k)\subset Q_{k,i}\subset B(x_{k,i},\kappa_1\eta^k)
\end{equation*}
for each $k\geqslant 1, 1\leqslant i\leqslant M(k)$.
Here $x_{k,i}$ is called the center of the cube $Q_{k,i}$.
Denote $B(x_{k,i},\kappa_1\eta^k)$ by $B(Q_{k,i})$ for each $k$ and $i$.
By Hyt\"onen and Kairema's result (\cite[Theorem 4.1]{Hytonen2012}), we have the following result on metric space $(\B_n,\rho)$:
\begin{theorem}\label{dyadic-thm}
There is a dyadic system $\mathscr{D}$ associated with the triple $(1/96,1/12,4)$ on $(\B_n,\rho)$.
\end{theorem}
\begin{remark}\label{remark2.1}
	We observe that for each $Q_{k,i}$ there exists a collection of sets $\{Q_{k+1,j}\}$ such that 
	\[Q_{k,i}=\bigcup\limits_{j}Q_{k+1,j},\]
	and each such cube $Q_{k+1,j}$ is called child of $Q_{k,i}.$
	Additionally, there exists a constant $C_1>0$ satisfying 
	\[\nu(Q_{k,i})\leqslant C_1\nu(Q_{k+1,j})\]
	for each cube $Q_{k,i}\in\mathscr{D}$ and its child $Q_{k+1,j}$.
\end{remark}

Now we introduce the Calder\'on-Zygmund decomposition on $\B_n$ associated with the dyadic system $\mathscr{D}$.
\begin{lemma}\label{C-Z}
	For $f\in L^1(\B_n,d\nu)$ and $t>0$, there exists a decomposition of $\B_n$ such that
	\begin{enumerate}[{(i)}]
		\item $\B_n=F\cup\Omega,$ with $ F\cap\Omega=\emptyset$;
		\item $|f(x)|\leqslant t$ on $F$ almost everywhere;
		\item $\Omega$ can be expressed as a union of pairwise disjoint cubes from the dyadic system $\mathscr{D}$ in Theorem \ref{dyadic-thm}. 
		That is, there exists a collection $G$ of such cubes such that
		\[\Omega=\bigcup\limits_{Q\in G} Q.\]
		And for each $Q\in G$, we have
		\[t<\frac{1}{\nu (Q)}\int_{Q}|f(x)|d\nu (x)<C_1t.\]
	\end{enumerate}
	Here, $C_1$ is the constant in Remark \ref{remark2.1} and independent of $f$ and $t$.
\end{lemma}
For the proof, one can see \cite[\S 1, 3.2]{Stein} or \cite[\S 5.3]{GTM}.

\section{The proof of Theorem \ref{th1.1}}
For a fixed $f\in L^1(\B_n,d\nu)$, the inequality \eqref{equ1.2} holds for all $t\leqslant||f||_{L^1(\B_n,d\nu)}$ since
\[(P  f)_*(t)=\nu (\{x\in\B_n:|(P  f)(x)|>t\})\leqslant 1 \leqslant\frac{||f||_{L^1(\B_n,d\nu)}}{t}.\]
Thus, it suffices to consider the case $t\geqslant||f||_{L^1(\B_n,d\nu)}.$ 

For a fixed $f\in L^1(\B_n,d\nu),$ we decompose $f$ into a ``good'' function $g$ and a ``bad'' function $b$ by Lemma \ref{C-Z}. Let 
\begin{equation*}
    g(x) = \left\{
    \begin{aligned}
	    &f(x)&& x\in F ,\\
	    &\frac{1}{\nu (Q_j)}\int_{Q_j}f(x)d\nu (x)\quad&&x\in Q_j\subset \Omega.
    \end{aligned} \right.
\end{equation*}
Then $g$ is bounded by Lemma \ref{C-Z}. Define $b(x)=f(x)-g(x)$. Then 
\[\{x\in\B_n:|P  f(x)|>t\}\subset\{x\in\B_n:|P  g(x)|>\frac{t}{2}\}\cup\{x\in\B_n:|P  b(x)|>\frac{t}{2}\}.\]
It suffices to prove 
\[(P  g)_*(t)\leqslant\frac{C}{t}||f||_{L^1(\B_n,d\nu)}\quad\mbox{and}\quad (P  b)_*(t)\leqslant\frac{C}{t}||f||_{L^1(\B_n,d\nu)}.\]

\textit{The estimate of $Pg$}:

By Lemma \ref{C-Z}, we have
\begin{equation*}
	\begin{aligned}
		||g||_{L^2(\B_n,d\nu)}^2&=\int_{F}|g(x)|^2d\nu (x)+\int_{\Omega}|g(x)|^2d\nu (x)\\
		&\leqslant t\int_{F}|f(x)|d\nu +\sum\limits_{Q\in G}\int_{\Omega}\left|\frac{1}{\nu (Q)}\int_{Q}f(x)d\nu (x)\right|^2\chi_Q(x)d\nu (x)\\
		&\leqslant t||f||_{L^1(\B_n,d\nu)}+C_1 t\sum\limits_{Q\in G}\int_{Q}\left|f(x)\right|d\nu (x)\\
		&\leqslant t||f||_{L^1(\B_n,d\nu)}+C_1 t\int_{\Omega}|f(x)|d\nu (x)\\
		&\leqslant(C_1+1)  t||f||_{L^1(\B_n,d\nu)}
	\end{aligned}
\end{equation*}
Since $P$ is bounded on $L^2(\B_n,d\nu)$, $P$ is of weak-type (2,2). 
Hence, we have
\[(P  g)_*(t)\leqslant\frac{1}{t^2}\int_{\B_n}|g(x)|^2d\nu (x)\leqslant \frac{(C_1+1)}{t}||f||_{L^1(\B_n,d\nu)}.\]

\textit{The estimate of $Pb$}:

The following theorem from \cite{Ureyen2023} is required to estimate the distribution of the function $P b$.

\begin{theorem}\label{th2.1}
	There exists a constant $C_2>0$ depending only on $n$ such that for all $x,y\in\B_n$,
	\begin{enumerate}[{(i)}]
		\item $|\mathcal{R} (x,y)|\leqslant\frac{C_2}{[x,y]^{n}},$
		\item $|\nabla_x\mathcal{R} (x,y)|\leqslant\frac{C_2}{[x,y]^{n+1}}.$
	\end{enumerate}
	Here $[x,y]^2=|x-y|^2+(1-|y|^2)(1-|x|^2)$ and $\nabla_x$ denotes the gradient operator with respect to the variable $x$.
\end{theorem}

We denote the cubes in $G$ by $\{Q_j\}_{j=1}^\infty$.
Then $\Omega=\bigcup\limits_{j=1}^\infty Q_j$.
For $Q_j\in G$, let $b_j(x)=b(x)\chi_{Q_j}(x)$. Then $b(x)=\sum\limits_{j=1}^\infty b_j(x)$ and
\[
	\int_{Q_j}b(x)d\nu (x)=0.
\]
For each $j$, set $y_j$ is the center of the cube $Q_j$.
Then we have
\begin{equation*}
	\begin{aligned}
		|(P  b_j)(x)|&=\left|\int_{Q_j}\mathcal{R} (x,y)b(y)d\nu (y)\right|\\
		&=\left|\int_{Q_j}(\mathcal{R} (x,y)-\mathcal{R} (x,y_j))b(y)d\nu (y)\right|\\
		&\leqslant\int_{Q_j}|\mathcal{R} (x,y)-\mathcal{R} (x,y_j)| |b(y)|d\nu (y)\\
	\end{aligned}
\end{equation*}
Recall the notion of $B(Q_j)$ introduced in the discussion at the beginning of Section 2.
Let $B_j=2B(Q_j)$ and $\Omega'=\bigcup\limits_{j}B_j$. 
By the definition of $B(Q_j)$ and Lemma \ref{C-Z}, there exists a constant $C_3$ independent of $f$ and $t$ such that 
\[\nu (\Omega')
\leqslant \sum\limits_{j=1}^\infty \nu (B_j)
\leqslant C_3 \sum\limits_{j=1}^\infty \nu (Q_j)
\leqslant \frac{C_3}{t}\sum\limits_{j=1}^\infty\int_{Q_j}|f(x)|d\nu (x)
\leqslant \frac{C_3}{t}||f||_{L^1(\B_n,d\nu)}.\]
Then we have
\begin{equation}\label{**}
	\begin{aligned}
		\int_{\B_n\backslash\Omega'}|(P  b)(x)|d\nu (x)\leqslant&\sum\limits_{j=1}^\infty\int_{\B_n\backslash B_j}|(P  b_j)(x)|d\nu (x)\\
		\leqslant&\sum\limits_{j=1}^\infty\int_{Q_j}|b(y)|\int_{\B_n\backslash B_j}|\mathcal{R} (x,y)-\mathcal{R} (x,y_j)|d\nu (x)d\nu (y)\\
	\end{aligned}
\end{equation}

\begin{claim}\label{claim2.1}
	There exists a constant $C_4$ depending only on $n$ such that
	\[\int_{\B_n\backslash B_j}|\mathcal{R} (x,y)-\mathcal{R} (x,y_j)|d\nu (x)\leqslant C_4\]
	for each $j\in\N$.
\end{claim}

\begin{proof}
By the mean value theorem and Theorem \ref{th2.1}, there exists a point $\bar{y_j}$ on the straight-line segment connecting $y_j$ and $y$ such that
\[|\mathcal{R} (x,y)-\mathcal{R} (x,y_j)|=|\nabla_y\mathcal{R} (x,\bar{y_j})| |y-y_j|\leqslant\frac{C_2}{[x,\bar{y_j}]^{n+1}}|y-y_j|.\]
Then 
\[
	\int_{\B_n\backslash B_j}|\mathcal{R} (x,y)-\mathcal{R} (x,y_j)|d\nu (x)\leqslant\int_{\B_n\backslash B_j}\frac{C_2}{[x,\bar{y_j}]^{n+1}}|y-y_j|d\nu(x).
\]
Note that $x\in\B_n\backslash B_j$ and $\bar{y_j}\in Q_j$, we have
\[|x-y_j|\geqslant 2|y_j-\bar{y_j}|.\]
Since 
\[
	[x,\bar{y_j}]^2=|x-\bar{y_j}|^2+(1-|\bar{y_j}|^2)(1-|x|^2),
\]
we have 
\[[x,\bar{y_j}]\geqslant |x-\bar{y_j}|\geqslant|x-y_j|-|y_j-\bar{y_j}|\geqslant\frac{1}{2}|x-y_j|.\]
Therefore, 
\[\int_{\B_n\backslash B_j}\frac{|y-y_j|}{[x,\bar{y_j}]^{n+1}}d\nu(x)\leqslant 2^{n+1}\int_{|x-y_j|\geqslant 2|y-y_j|}\frac{|y-y_j|}{|x-y_j|^{n+1}}d\nu(x)\leqslant C_4\]
for some constant $C_4=C_4(n)$.
This completes the proof of the claim.
\end{proof}

By (\ref{**}) and Claim \ref{claim2.1}, we have
\begin{equation}
	\begin{aligned}
		\int_{\B_n\backslash\Omega'}|(P  b)(x)|d\nu (x)&\leqslant C_4\sum\limits_{j=1}^\infty\int_{Q_j}|b(y)|d\nu (y)\\
		&=C_4\int_{\Omega}|b(y)|d\nu (y).
	\end{aligned}
\end{equation}
Since $|b(y)|=|f(y)-g(y)|\leqslant|f(y)|+|g(y)|\leqslant|f(y)|+C_1t,$ we have
\[\int_{\Omega}|b(y)|d\nu (y)\leqslant\int_{\Omega}|f(y)|d\nu (y)+C_1t\nu (\Omega).\]
By Lemma \ref{C-Z}, we have
\[\nu (\Omega)=\sum\limits_{j=1}^\infty\nu (Q_j)\leqslant\frac{1}{t}\sum\limits_{j=1}^\infty\int_{Q_j}|f(x)|d\nu (x)\leqslant\frac{1}{t}||f||_{L^1(\B_n,d\nu)}.\]
Therefore,
\[\int_{\B_n\backslash\Omega'}|(P  b)(x)|d\nu (x)\leqslant C_4\int_{\Omega}|b(y)|d\nu (y)\leqslant C_4(1+C_1)||f||_{L^1(\B_n,d\nu)}.\]
Moreover,
\[(P  b)_*(t)\leqslant \nu \big((\B_n\backslash\Omega')\cap\{|(P  b)(x)|>t\}\big)+\nu (\Omega')\leqslant\frac{C_4(1+C_1)+C_3}{t}||f||_{L^1(\B_n,d\nu)}.\]
This completes the proof of Theorem \ref{th1.1}.

\section{Acknowledgments}
I am grateful to Zipeng Wang for his guidance and valuable discussions. 
I would also like to thank reviewers for their careful reading and valuable comments, which greatly enrich the background and context of this work.

This work was completed while the author was visiting Chongqing University. 
The author gratefully acknowledges the financial support from the Department of Mathematics and the Mathematics Center at Chongqing University. 
Special thanks are extended to Yi Wang and Fugang Yan for their warm hospitality. 

This work was supported by the National Natural Science Foundation of China (Grant No. 12231005) and National Key R\&D Program of China (2024YFA1013400).


\begin{thebibliography}{10}
\bibitem{ABC1996}
P. Ahern, J. Bruna, C. Cascante,
\newblock {{$H^p$}-theory for generalized {$M$}-harmonic functions in the unit ball},
\newblock {\em Indiana Univ. Math. J.}, 45(1996), no.1, 103-135.

\bibitem{BCEM}
C. Bender, D. Chakrabarti, L. Edholm, M. Mainkar,
\newblock {$L^p$-regularity of the Bergman projection on quotient domains},
\newblock {\em Canad. J. Math.}, 74(2022), no.3, 732-772.

\bibitem{CE}
D. Chakrabarti, L. Edholm,
\newblock {Projections onto $L^p$-Bergman spaces of Reinhardt domains},
\newblock {\em Adv. Math.}, 451(2024), Paper No. 109790, 46 pp.

\bibitem{CH}
D. Chakrabarti, Z. Huo,
\newblock {Restricted type estimates on the Bergman projection of some singular domains},
\newblock {\em J. Geom. Anal.}, 35(2025), no.10,  Paper No. 302, 36 pp.

\bibitem{CK1}
A. Christopherson, K. Koenig, 
\newblock {Weak-type regularity of the Bergman projection on rational Hartogs triangles},
\newblock {\em Proc. Amer. Math. Soc.}, 151(2023), no.4, 1643-1653.

\bibitem{CK2}
A. Christopherson, K. Koenig, 
\newblock {Weak-type regularity of the Bergman projection on generalized Hartogs triangles in $\C^3$},
\newblock {\em Complex Var. Elliptic Equ.}, 69(2024), no.10,  1723-1738.

\bibitem{CK3}
A. Christopherson, K. Koenig, 
\newblock {Endpoint behavior of the Bergman projection on a class of $n$-dimensional Hartogs triangles},
\newblock {\em Complex Anal. Oper. Theory}, 18(2024), no.9, Paper No. 180, 9 pp.

\bibitem{CZ}
D. Chakrabarti, Y. E. Zeytuncu,
\newblock {$L^p$ mapping properties of the Bergman projection on the Hartogs triangle},
\newblock {\em Proc. Amer. Math. Soc.}, 144(2016), no.4,  1643-1653.

\bibitem{DHZZ2001}
Y. Deng, L. Huang, T. Zhao, D. Zheng,
\newblock {Bergman projection and {B}ergman spaces},
\newblock {\em J. Operator Theory}, 46(2001), no.1, 3-24.

\bibitem{EM1}
L. D. Edholm, J. D. McNeal,
\newblock {The Bergman projection on fat Hartogs triangles: $L^p$ boundedness.},
\newblock {\em Proc. Amer. Math. Soc.}, 144(2016), no.5, 2185-2196.

\bibitem{EM2}
L. D. Edholm, J. D. McNeal,
\newblock {Bergman subspaces and subkernels: degenerate $L^p$ mapping and zeroes},
\newblock {\em J. Geom. Anal.}, 27(2017), no.4, 2658-2683.

\bibitem{Englis2024}
M. Engli\v s, and E. Youssfi,
\newblock {{$M$}-harmonic reproducing kernels on the ball},
\newblock {\em J. Funct. Anal.}, 286(2024), no.1, Paper No. 110187.

\bibitem{Jaming2004}
S. Grellier, P. Jaming,
\newblock {Harmonic functions on the real hyperbolic ball. {II}. {H}ardy-{S}obolev and {L}ipschitz spaces},
\newblock {\em Math. Nachr.}, 268(2004), 50-73.

\bibitem{GTM}
L. Grafakos
\newblock {\em Classical Fourier Analysis},
\newblock Graduate Texts in Mathematics 249, 3rd ed., Springer, New York, 2014.

\bibitem{H}
H. Hedenmalm,
\newblock The dual of a Bergman space on simply connected domains, 
\newblock {\em J. Anal. Math.}, 88 (2002), 311-335.

\bibitem{HW}
Z. Huo, D. B. Wick, 
\newblock {Weak-type estimates for the Bergman projection on the polydisc and the Hartogs triangle},
\newblock {\em Bull. Lond. Math. Soc.}, 52(2020), no.5, 891-906.

\bibitem{Hytonen2012}
T. Hyt\"onen and A. Kairema
\newblock Systems of dyadic cubes in a doubling metric space, 
\newblock {\em Colloq. Math.}, 126(2012), no.1, 1-33.

\bibitem{Jaming1999}
P. Jaming,
\newblock {Harmonic functions on the real hyperbolic ball. {I}. {B}oundary values and atomic decomposition of {H}ardy spaces},
\newblock {\em Colloq. Math.}, 80(1999), no.1, 63-82.

\bibitem{KW}
K. Koenig, Y. Wang
\newblock {Harmonic Bergman theory on punctured domains},
\newblock {\em J. Geom. Anal.}, 31(2021), no.7, 7410-7435.

\bibitem{LiuShi2003}
C. Liu, J. Shi,
\newblock {Invariant mean-value property and {$\mathcal{M}$}-harmonicity in the unit ball of {$\R^n$}},
\newblock {\em Acta Math. Sin. (Engl. Ser.)}, 19(2003), no.1, 187-200.

\bibitem{Rudin}
W. Rudin,
\newblock {\em Function theory in the unit ball of $\C^n$},
\newblock {Grundlehren Math.Wiss.}, vol. 241,
Springer, New York, 1980.

\bibitem{S1}
A. A. Solov'ev, 
\newblock {Estimates in $L^p$ of the integral operators that are connected with spaces of analytic and harmonic functions. (Russian)},
\newblock {\em Dokl. Akad. Nauk SSSR}, 240(1978), no.6, 1301-1304.

\bibitem{S2}
A. A. Solov'ev, 
\newblock {Estimates in $L^p$ of integral operators connected with spaces of analytic and harmonic functions. (Russian)},
\newblock {\em Sibirsk. Mat. Zh.}, 26(1985), no.3, 168-191, 226.

\bibitem{Stoll1}
M. Stoll,
\newblock {\em invariant potential theory in the unit ball of $\C^n$},
\newblock {Lond. Math. Soc. Lect.} Note Series, vol. 199, Cambridge University Press, Cambridge, 1994.

\bibitem{Stoll2012}
M. Stoll,
\newblock {Weighted {D}irichlet spaces of harmonic functions on the real hyperbolic ball},
\newblock {\em Complex Var. Elliptic Equ.}, 57(2012), no.1, 63-89.

\bibitem{Stoll2019}
M. Stoll,
\newblock {The reproducing kernel of {$\mathcal{H}^2$} and radial eigenfunctions of the hyperbolic {L}aplacian},
\newblock {\em Math. Scand.}, 124(2012), no.1, 81-101.

\bibitem{Stoll}
M. Stoll,
\newblock {\em Harmonic and Subharmonic Function Theory on the hyperbolic ball},
\newblock {Lond. Math. Soc. Lect.} Note Series, vol. 431, Cambridge University Press, Cambridge, 2016.

\bibitem{Stein}
E. M. Stein, 
\newblock {\em Singular integrals and differentiability properties of functions},
\newblock Princeton Mathematical Series, vol. 30, Princeton University Press, Princeton, NJ, 1970.

\bibitem{Ureyen2023}
A.E. \"Ureyen,
\newblock {$\mathcal{H}$}-harmonic {B}ergman projection on the real hyperbolic ball,
\newblock {\em J. Math. Anal. Appl.}, 519(2023), 126802.

\end{thebibliography}
\end{document}